\documentclass[11pt]{amsart}

\usepackage{amsthm}
\usepackage{amssymb}
\usepackage{graphicx}									 
\usepackage{stmaryrd}                  
\usepackage{footmisc}                  
\usepackage{paralist}
\usepackage{wrapfig}
\usepackage[draft]{pdfcomment}
\usepackage{bbm}
\usepackage[T1]{fontenc}
\usepackage[utf8]{inputenc}
\usepackage[arrow, matrix, curve]{xy}
\usepackage{color}

\usepackage{tikz-cd}

\usepackage{enumitem}

\newtheorem{thm}{Theorem}[section]

\newtheorem{lem}[thm]{Lemma}
\newtheorem{fact}[thm]{Fact}

\theoremstyle{definition}

\theoremstyle{remark}
\newtheorem{rem}[thm]{Remark}

\newcommand{\R}{\mathbb{R}}

\newcommand{\To}{\rightarrow}

\newcommand{\MTo}{\mapsto}

\title[Regularity of the geodesic flow on submanifolds]{Regularity of the geodesic flow \\ on submanifolds}

\author[C. Lange]{Christian Lange}
\address{Ludwig-Maximilians-Universit\"at M\"unchen, Mathematisches Institut\newline\indent Theresienst. 39, 80333 München, Germany}
\email{lange@math.lmu.de, clange.math@gmail.com}

\begin{document}

\begin{abstract}
We show that the geodesic flow and the exponential map of a $\mathcal{C}^k$ submanifold of $\R^n$ with $k\geq 2$ are of class $\mathcal{C}^{k-1}$.
\end{abstract}

\maketitle

\section{Introduction}

The geodesic flow of a Riemannian manifold is a widely studied dynamical system with interesting properties. 
It is closely related to the exponential map, which is an important tool in Riemannian geometry and its applications. 
For a smooth Riemannian manifold both the geodesic flow and the exponential map are smooth maps by standard ODE theory. However, sometimes only less regularity is available. Properties related to geodesics of general Riemannian metrics with limited regularity have been studied by several authors, see e.g. \cite{AY06,KSS14,Min15,ST18}. Here we are concerned with the existence and regularity of the geodesic flow and the exponential map for submanifolds of limited regularity, cf. \cite{Ha50}.

For a Riemannian metric $g$ of class $\mathcal{C}^{k}$ on some manifold $M$ the Christoffel symbols, which appear as coefficients of the corresponding geodesic equation, are of class $\mathcal{C}^{k-1}$. For $k\geq 2$ the Christoffel symbols are in particular Lipschitz and then solutions of the geodesic equation with given initial conditions exist uniquely by the Picard--Lindelöf theorem \cite{Ha82}. Moreover, they define constant speed locally length minimizing curves. In this case their maximal extensions determine the geodesic flow $\phi: V\subset \R \times TM \rightarrow TM$ and the exponential map $\exp: W \subset TM  \rightarrow M$, $\exp(v)=\phi(1,v)$, of $g$, both of which are of class $\mathcal{C}^{k-1}$ by Peano's theorem on the dependence on initial conditions \cite[Theorem~V.4.1]{Ha82}. For a $\mathcal{C}^{1,1}$ metric solutions of the geodesic equation still give rise to a locally Lipschitz geodesic flow and are locally length minimizing among absolutely continuous curves \cite{Min15}, cf. \cite{KSS14}. In lower regularity these statements fail \cite{ST18}. Moreover, in general the regularity of the geodesic flow and the exponential map is not better than that \cite[Example~2.3]{DK81}\cite{Sa93}.

On the other hand, we show that there is a well-defined geodesic flow on a $\mathcal{C}^{2}$ \emph{sub}manifold $M$ of $\R^n$ (with $\mathcal{C}^{1}$ metric) defined in terms of constant speed locally length minimizing curves, and that the geodesic flow of a $\mathcal{C}^{k}$ submanifold of $\R^n$ for $k\geq 2$ is as regular as the tangent bundle $TM$ on which it is defined on.

\begin{thm}\label{thm:regularity_gf} Let $M$ be a $\mathcal{C}^{k}$ submanifold of $\R^n$ for some $k\geq 2$. Then for any $v \in TM$ there exists a unique maximal locally length minimizing solution $\gamma_v$ of the geodesic equation with $\gamma'(0)=v$. Let $\phi: V \rightarrow TM$, $V \subset \R \times TM$, $\phi(t,v):= \gamma_v'(t)$ be the geodesic flow defined at those $(t,v) \in \R \times TM$ for which $\gamma_v$ extends up to time $t$. Then $V \subset \R \times TM$ is open and $\phi$ is of class $\mathcal{C}^{k-1}$. In particular, also the exponential map $\exp:  W\subset TM \rightarrow M$, $W= \{1\} \times TM \cap V$, is of class $\mathcal{C}^{k-1}$.
\end{thm}

Under the assumptions of the theorem locally length minimizing constant speed curves on $M$ are of class $\mathcal{C}^{k}$ and satisfy the geodesic equation \cite{ST18}. In case of a $\mathcal{C}^{2}$ surface in $\R^3$ Hartman \cite[Theorem~1]{Ha50} moreover showed that solutions of the geodesic equations are uniquely determined by their initial conditions. Theorem \ref{thm:regularity_gf} generalizes a regularity result by Hartman for the exponential map with fixed base point of surfaces \cite[Theorem~2]{Ha50}. The example \cite[Example~2.3]{DK81}\cite{Ha51} combined with Nash's embedding theorem \cite{Na56} shows that Theorem \ref{thm:regularity_gf} is optimal in terms of the regularity of the geodesic flow and the exponential map. Moreover, by Nash's theorem the statement of Theorem \ref{thm:regularity_gf} also holds for $\mathcal{C}^{k}$ submanifolds, $k\geq 3$, of Riemannian manifolds with a Riemannian metric of class $\mathcal{C}^{k}$. Regularity statements are for instance relevant in the context of Taylor expansions, cf. \cite{MMS14}.

The proof of Theorem \ref{thm:regularity_gf} is based on the observation that the curvature of $M$ is more regular than on a general Riemannian manifold due to Gauß' equation, cf. \cite{AK21}. This observation combined with a regularity result for Riemannian metrics due to DeTurck--Kazdan readily implies that for $k\geq 3$ and $0<\alpha<1$ the geodesic flow and the exponential map of a $\mathcal{C}^{k,\alpha}$ submanifold $M$ of $\R^n$ are of class $\mathcal{C}^{k-1,\alpha}$. Namely, in this case they show that there are harmonic coordinates that depend in a $\mathcal{C}^{k,\alpha}$ way on the original coordinates \cite[Lemma~2.1]{DK81} with respect to which the Riemannian metric is of class $\mathcal{C}^{k,\alpha}$ \cite[Theorem~4.5]{DK81}. However, for $\alpha=0,1$ elliptic regularity theory cannot be applied and the statement about harmonic coordinates does not hold \cite{SS76}.

Instead, in these cases we use the fact that the derivatives of the exponential map can be expressed as solutions of Jacobi's equation whose coefficients are determined by the curvature. Standard ODE theory then implies that the derivatives of the exponential map, and then also the exponential map itself, inherit the enhanced regularity, see Section \ref{sub:jacobi_field}. However, expressing the derivatives of the exponential map as solutions of Jacobi's equation requires caution in low regularity where the usual assumptions of Schwarz' theorem about the interchangeability of partial derivatives can fail. For $k=3$ this issue can be handled with a version of Schwarz' theorem due to Peano, see Section \ref{sub:int_part_deri}.

For $k=2$ we instead work with the Jacobi equation of smooth approximations of our submanifold in order to avoid issues related to exchanging partial derivatives. This allows us to conclude that the geodesic flows of the approximating submanifolds subconverge in $\mathcal{C}^1$ to a flow on the limit $\mathcal{C}^2$ submanifold. Based on the local curvature bounds, we then show that locally length minimizing curves converge to locally length minimizing curves on the limit space, ensuring that the limit flow is actually a geodesic flow. Moreover, since locally length minimizing curves on the limit space cannot branch because of the curvature bound, also the uniqueness statement follows.

In a similar way one can complete a proof of the of following statement: for a $\mathcal{C}^{k,\alpha}$ submanifold $M$ of $\R^n$ with $k\geq 1$, $\alpha \in [0,1]$ and $k+\alpha \geq 2$ there is a unique (maximal) geodesic flow defined in terms of constant speed locally length minimizing curves and it is of class $\mathcal{C}^{k-1,\alpha}$, see Lemma \ref{lem:geo_lipschitz} and Remark \ref{rem:hoelder_case}. Alternatively, one could phrase this by saying that the Filippov geodesic equation has unique locally length minimizing solutions and that the geodesic flow defined in terms of these solutions is of class $\mathcal{C}^{k-1,\alpha}$, cf. the the remark after Lemma \ref{lem:geo_lipschitz}.

In even lower regularity, e.g. on a $\mathcal{C}^{1}$ submanifold, the geodesic flow is in general not everywhere defined let alone continuous. For instance, embedding a double of the strange billiard table constructed in \cite{Hal77} via Nash's $C^1$ embedding theorem \cite{Na54} into some $\R^n$ yields a $\mathcal{C}^1$ submanifold with curvature bounded from below in the Alexandrov sense which does not even admit a continuous quasigeodesic flow, cf. \cite{La22}.
\newline
\newline
\emph{Acknowledgements.} The author thanks Alexander Lytchak for helpful suggestions and comments that lead to a simpler proof. He thanks Alberto Abbondandolo for discussions about elliptic regularity.

\section{Preliminaries and proof in high regularity}\label{sec:high_reg}

In this section we recall some notions and prove Theorem \ref{thm:regularity_gf} for $k\geq 3$. For more background on Riemannian geometry and refer the reader to e.g. \cite{Wa83,dC92}.

\subsection{Geodesic flow and exponential map}\label{sub:geodesic_flow}
In the following, by a \emph{geodesic} we mean a constant speed locally length minimizing curve. If the Riemannian metric of $M$ is of class $\mathcal{C}^{1,1}$, then a  geodesic $\gamma$ on $M$ satisfies the geodesic equation $\nabla_{\gamma'(t)} \gamma'(t)=0$ \cite{ST18}. Recall that in a coordinate chart $(U,x=(x_1,...,x_n))$ in a neighborhood of $\gamma(t_0)$ the geodesic equation can be expressed as
\begin{equation}\label{eq:geodesic_equation}
	0=\nabla_t \partial_t \gamma_v(t)= \sum_{k=1}^n\left(\frac{d^2x_k}{dt^2}+ \sum_{i,j=1}^n\Gamma_{ij}^k \frac{dx_i}{dt}\frac{dx_j}{dt}\right) \frac{\partial}{\partial x^k}.
\end{equation}
Here $\Gamma_{ij}^k$ are the Christoffel symbols which are determined by the coordinate vector fields $\partial_k=\frac{\partial}{\partial x^k}$ and the Levi-Civita connection $\nabla$ of $M$ through the relation
\[
			\nabla_{\partial_i} \partial_j = \sum_{k=1}^n \Gamma_{ij}^k \partial_k,
\]
and which can be expressed in terms of the metric coefficients and its first derivatives, see e.g. \cite{dC92}. Any differentiable curve $\gamma$ in $M$ determines a curve $(\gamma(t),\gamma'(t))$ in $TM$. With respect to the coordinates on $TU$ induced by $(U,x)$ the system (\ref{eq:geodesic_equation}) becomes equivalent to the first order system
\begin{equation}\label{eq:geodesic_equation_fos}
\begin{cases}
\frac{dx_k}{dt}=y_k,\\
\frac{dy_k}{dt}=-\sum_{i,j=1}^n\Gamma_{ij}^k y_i y_j.\\
\end{cases}
\end{equation}
If the Riemannian metric is of class $\mathcal{C}^{k-1}$ for some $k\geq 3$, then the Christoffel symbols are in particular locally Lipschitz continuous. In this case solutions of the geodesic equation are uniquely determined by their initial conditions in $TM$ due to the Picard-Lindelöf theorem, and their maximal extensions determine the geodesic flow $\phi: V \subset \R \times TM \rightarrow TM$. Moreover, the system (\ref{eq:geodesic_equation_fos}) implies that geodesics are of class $\mathcal{C}^{k}$ as a map to $M$ and that the geodesic flow is of class $\mathcal{C}^{k-2}$ by Peano's theorem on the dependence on initial conditions \cite[Theorem~V.4.1]{Ha82}. Hence, the same is true for the exponential map $\exp(v)=\pi(\phi(1,v))$, where $\pi: TM \To M$ denotes the canonical projection.

In Section \ref{lem:geo_lipschitz} we will show that for a $\mathcal{C}^2$ submanifold of $\R^n$ there still exists a unique (maximal) geodesic flow defined in terms of constant speed locally length minimizing curves. In this case geodesics still satisfy the geodesic equation and are of class $\mathcal{C}^1$ by the Du Bois-Reymond trick \cite[Section~2.1]{ST18}. In fact, these statements also hold for a $\mathcal{C}^{1,1}$ submanifold if the geodesic equation is understood in a weaker sense of Filippov, see Lemma \ref{lem:geo_lipschitz}  and the remark thereafter.

\subsection{Differentiability}\label{sub:differentiability} For $k\geq 3$ in the situation of Theorem~\ref{thm:regularity_gf} the geodesic flow exists, is continuously differentiable and satisfies all the usual properties by standard ODE theory, see Section~\ref{sub:geodesic_flow}. In this case, in view of proving the regularity claim of Theorem~\ref{thm:regularity_gf}, it suffices to work locally, i.e. we can always replace $M$ by a small neighborhood of a point $p\in M$, by continuity and the dynamical system property $\phi(s+t,v)=\phi(s,\phi(t,v))$ (which hold wherever these expressions are defined) of the geodesic flow as the composition of two $\mathcal{C}^k$ maps is again of class $\mathcal{C}^k$. In the case $k = 2$ we will be able to do the same once we have shown that Theorem~\ref{thm:regularity_gf} holds locally, see Section \ref{sec:low_reg_case}.

Although it suffices to work locally, let us also recall the following coordinate free picture. For a $\mathcal{C}^{k}$ manifold with $\mathcal{C}^{k-1}$ Riemannian metric, $k\geq 2$, let $E$ be the pullback of the Whitney sum vector bundle $TM \oplus TM \rightarrow M$ to $TM$:
\[
\begin{tikzcd}
E \arrow[dashed]{r}{\pi_1} \arrow[swap,dashed]{d}{\pi_2} & TM \oplus TM \arrow{d}{\pi\oplus \pi} \\
TM \arrow{r}{\pi} & M
\end{tikzcd}
\]
The pullback diagram induces a $\mathcal{C}^{k-2}$ vector bundle structure on $E$. We represent an element $X \in T_vTM$ as the derivative of some $\mathcal{C}^{k-1}$ curve $v(s)\in T M \subset TM$ with $v(0)=v$ at $s=0$. Set $p(s)=\pi(v(s))$. Then the map
\[
		 \begin{array}{cccl}
		 \psi : &TTM				& \To  	&	E \\
		           		& X 	& \MTo  & (\pi(X),(D\pi X = p'(0), (\nabla_{p'(0)} v)(0)))
		 \end{array}
\]
defines a $\mathcal{C}^{k-2}$ vector bundle isomorphism.

We also recall the following characterization of $\mathcal{C}^{k}$ differentiability.

\begin{fact}\label{lem:diff_char}
Let $N$ and $P$ be $\mathcal{C}^{k}$ manifolds with $k\geq 1$. Then a continuous map $\varphi: N \rightarrow P$ is of class $\mathcal{C}^{k}$, if and only if each point $p$ in $N$ has a neighborhood $U$ such that for each $\mathcal{C}^{k-1}$ vector field $X$ on $U$ the map $U \ni q \mapsto D_{X(q)}\varphi(q)\in TP$ exists and is of class $\mathcal{C}^{k-1}$.
\end{fact}

We would like to apply this statement in the situation of Theorem~\ref{thm:regularity_gf} to the geodesic flow $\phi: V=N \subset \R \times TM \rightarrow TM=P$ of a $\mathcal{C}^{k}$ submanifold $M$ with $k\geq 2$, so that $N$ and $P$ are $\mathcal{C}^{k-1}$ manifolds, once we know that the geodesic flow exists and is continuous. By linearity it then suffices to check the condition in Fact~\ref{lem:diff_char} for vector fields on $V\subset \R \times TM$ which are everywhere tangent to the second factor $TM$, and for the constant vector field $\frac{\partial}{\partial t}$ which is everywhere tangent to the second factor. 

If the geodesic flow in the situation of Theorem~\ref{thm:regularity_gf} exists and is continuous, then its geodesics are of class $\mathcal{C}^{k}$ and the geodesic vector field $G(v)=\frac{\partial}{\partial t}|_{t=0}\phi(t,v)$ is of class $\mathcal{C}^{k-2}$ by the geodesic equation. In the coordinate free picture the latter follows from $\pi_2(\psi(G(v)))=v$ and $\pi_1(\psi(G(v)))=(v,0)$. Hence, for the vector field $\frac{\partial}{\partial t}$ on $V$ the condition of Fact~\ref{lem:diff_char} is satisfied in this case.

\subsection{Jacobi fields and Gauß equation}\label{sub:jacobi_field}

For a proof of Theorem \ref{thm:regularity_gf} in the case $k \geq 3$ it remains to verify the condition in Fact~\ref{lem:diff_char} for vector fields on $V\subset \R \times TM$ which are everywhere tangent to the second factor $TM$.  For such a vector field $X$ and sufficiently large $k$ it is well known that $\psi(D\phi(v,t) X(v))$ can be obtained as a solution of Jacobi's equation with initial conditions $(J(0),J'(0))=\pi_2(\psi(X))$ along the geodesic $\gamma_v:[0,1] \rightarrow M$ with initial condition $v=\pi_1(\psi(X))$. We recall the computation for the convenience of the reader. Consider the geodesic variation $\tau(t,s)=\exp(t\cdot v(s))=\pi(\phi(t,v(s)))$ of $\gamma_v$, where again $v(s)$ is a $\mathcal{C}^1$ curve in $TM$ with value $v$ and derivative $X(v)$ at $s=0$. Then the Jacobi field $J(t)=\frac{\partial}{\partial s}|_{s=0} \tau(t,s)$ along $\gamma_v$ satisfies the following so-called Jacobi equation, cf. e.g. \cite[Chapter~5]{dC92},
\begin{equation}\label{eq:jacobi_equation}
		J''(t) = \nabla_t^2  \partial_s \tau(t,s)|_{s=0} = \nabla_t \nabla_s \partial_t \tau(t,s)|_{s=0}=-R(\gamma_v'(t),J(t))\gamma_v'(t),
\end{equation}
where $R$ denotes the Riemann curvature tensor. Here we have interchanged derivatives twice, see \cite[Lemma~3.4, Lemma~4.1]{dC92}, and applied the geodesic equation (\ref{eq:geodesic_equation}) in the last step. Moreover, we have
\[
		\pi_1(\psi(D\phi(t,v) X(v)))= \phi(t,v)=\gamma_v'(t),
\]
\[
D\pi (D\phi(t,v) X(v))= \partial_s \pi(\phi(t,v(s)))|_{s=0} = \partial_s \tau(t,s)|_{s=0} = J(t),
\]
\begin{equation}\label{eq:init_jac_deriv}
\nabla_{J(t)} \phi(t,v(s)) |_{s=0} =  \nabla_s \partial_t \tau(r,s) |_{s=0} = \nabla_t \partial_s \tau(t,s) |_{s=0} = J'(t),
\end{equation}
and thus
\begin{equation}\label{eq:init_jac}
\psi(D\phi(v,t) X(v,t))=(c_v'(t),(J(t),J'(t)))
\end{equation}
as claimed. However, note that interchanging derivatives in (\ref{eq:jacobi_equation}) and (\ref{eq:init_jac_deriv}) requires sufficiently high regularity to begin with. For now we assume that $k\geq 4$ so that this is not an issue.

At this point it suffices to verify that the ``right hand sides'' of the geodesic equation and the Jacobi equation are of class $\mathcal{C}^{k-2}$ as the regularity claim about the geodesic flow and the exponential map is then implied by the standard ODE result on the differentiable dependence on initial conditions \cite[Theorem~V.4.1]{Ha82}. For the geodesic equation this assumption is satisfied as the Christoffel symbols involve only first order derivatives of the metric. In case of the Jacobi equation the right hand side is given by the Riemann curvature tensor, which in general is only of class  $\mathcal{C}^{k-3}$. However, for a $\mathcal{C}^{k}$ submanifold $M$ of a Riemannian manifold $\overline{M}$ with a $\mathcal{C}^{k}$ Riemannian metric, $k\geq 3$, the curvature tensor $R$ of $M$ is actually of class $\mathcal{C}^{k-2}$ by the Gauß equation \cite[Proposition~3.1]{dC92}. In view of the case $k= 3$ we recall how the expression of Gauß' equation naturally comes up in the derivation of Jacobi's equation for submanifolds of $\R^n$. Note that we denote the usual derivative in $\R^n$ also by $D$ if we think of the object that we differentiate as a vector field: Set $V= \partial_t \tau$. Then we can write $D_s V = \nabla_s V + \Pi(\partial_s\tau,V)$, where the two terms of the sum, the covariant derivative and the second fundamental form $\Pi$, a $(0,2)$-tensor, are the components of $D_s V$ in the tangent space and in the normal space of $M$. Likewise, we have $D_t V = \Pi(\partial_t\tau,V)$ by the geodesic equation. Taking another derivative leads to
\[
		\left\langle D_t D_sV,W \right\rangle = \left\langle \nabla_t\nabla_s V,W\right\rangle + \left\langle \Pi(\partial_s \tau,V),\Pi(\partial_t \tau,W) \right\rangle
\]
and
\[
		\left\langle D_sD_tV,W \right\rangle =  \left\langle \Pi(\partial_t \tau,V),\Pi(\partial_s \tau,W) \right\rangle
\]
at $p=c(t)\in M$ for all $W\in T_pM$. Provided we are in sufficiently high regularity so that interchanging usual derivatives is not a problem, we obtain
\begin{equation}\label{eq:jac_eq_sub}
		\left\langle J'',W\right\rangle =  \left\langle \Pi(V,V),\Pi(J,W) \right\rangle - \left\langle \Pi(J,V),\Pi(V,W) \right\rangle
\end{equation}
for all $W\in T_pM$, where we abbreviate the second covarient derivative of $J$ along $c$ as $J''$. Here, the right hand side is indeed of class $\mathcal{C}^{k-2}$ and can be identified with $-\left\langle R(V,J)V,W \right\rangle$. 

\subsection{Interchanging partial derivatives}\label{sub:int_part_deri}

In order to complete the proof of Theorem \ref{thm:regularity_gf} for $k=3$ it suffices to justify the commutation of partial derivatives in (\ref{eq:init_jac_deriv}) and (\ref{eq:jac_eq_sub}) in this case. A nice overview on different versions of Schwarz' theorem can be found in the paper \cite{Min14} by Minguzzi. Here we can apply the following version due to Peano \cite{Pe92}\cite[Theorem~9.41]{Ru76}.

\begin{thm}[Peano]\label{thm:peano_schwarz} Let $f:U \To \R^n$ be defined on an open set $U\subset \R^2$. Suppose that the partial derivatives $D_1f$, $D_2D_1f$ and $D_2f$ exist on $U$ and that $D_2D_1f$ is continuous at some point $(x_0,y_0)\in U$. Then $D_1D_2f$ exists at $(x_0,y_0)$ and
\[
			(D_1D_2f)(x_0,y_0) = (D_2D_1f)(x_0,y_0).
\]
\end{thm}
Indeed, in case of a $\mathcal{C}^3$ submanifold $\tau$, $V=\partial_t \tau$ and $D_t V = \Pi(V,V)$ are a priori all of class $\mathcal{C}^1$. Therefore, the computions in (\ref{eq:init_jac_deriv}) and (\ref{eq:jac_eq_sub}) are justified in this case by Peano's theorem. This completes the proof of Theorem \ref{thm:regularity_gf} for $k\geq 3$.

\section{The low regularity case}\label{sec:low_reg_case}

For submanifolds of regularity below $\mathcal{C}^3$ it becomes more difficult to show that Jacobi fields satisfy the Jacobi equation. Instead, we will work with the Jacobi equations of approximating smooth submanifolds. 

Let $M$ be a $\mathcal{C}^2$ submanifold of $\R^n$. Again, once we have shown the claim of Theorem~\ref{thm:regularity_gf} locally and verified that geodesics and the resulting geodesic flow satisfy all the usual properties, notably the dynamical system property, the claim also follows globally as discussed in Section \ref{sub:differentiability}. Locally around a point $p\in M$ we write $M$, perhaps after the application of an isometry of $\R^n$, as a graph of a $\mathcal{C}^{2}$ function $h: \R^{m} \supset U \To \R^{n-m}$ with $0 \in U$ and $h(0)=p$. Perhaps after shrinking $U$ we approximate $h$ by a sequence of smooth functions $h_l : U \To \R^{n-m}$ in $\mathcal{C}^{2}$ with $h_l(0)=p$. We denote the pullback of the euclidean metric to $U$ via $h_l$ and $h$ as $g_l$ and $g$, respectively. Then $g_l$ converges to $g$ in $\mathcal{C}^{1}$ and the second fundamental forms $\Pi_l$ converge to the second fundamental form $\Pi$ of $M$ in $\mathcal{C}^{0}$. If $M$ is merely of class $\mathcal{C}^{1,1}$, then the $g_l$ converge to $g$ in $\mathcal{C}^{0,1}$ and the $\Pi_l$ are uniformly bounded.

For completeness we state the following lemma in terms of a $\mathcal{C}^{1,1}$ submanifold.

\begin{lem}\label{lem:geo_lipschitz}
Let $M$ be a $\mathcal{C}^{1,1}$ submanifold of $\R^n$. Then any point $p\in M$ has an open neighborhood $U_p$ such that for any $v \in TU_p$ there exists a unique maximal locally length minimizing constant speed $\mathcal{C}^{1}$ curve $\gamma_v$ in $U_p$ with $\gamma'(0)=v$. Let $\phi: V \rightarrow TU_p$, $V \subset \R \times TU_p$, $\phi(t,v):= \gamma_v'(t)$ be the geodesic flow defined at those $(t,v) \in \R \times TU_p$ for which $\gamma_v$ extends up to time $t$. Then $V \subset \R \times TU_p$ is open, $\phi$ is locally Lipschitz and $\phi(\cdot,v)$ is $\mathcal{C}^{1,1}$ wherever it is defined.
\end{lem}
\begin{proof} We keep working in the setting introduced above. We claim that the injectivity radii of the $g_l$ in some smaller neighborhood of $0$ are uniformly bounded from below by some $r>0$. On a compact smooth submanifold $N$ of $\R^n$ the injectivity radius is bounded from below by $\min \{\pi/C,l(N)/2\}$, where $C$ is an upper bound on the norms of the second fundamental forms of $N$ and $l(N)$ is the length of the shortest closed geodesic \cite[Section~6.6]{Pe16}. Since the $\Pi_l$ are uniformly bounded, we can use a smooth hat function to complete the graphs of the functions $h_l$ to compact submanifolds $N_l$ with a uniform bound on their second fundamental forms $C'$ without changing the graphs in a small neighborhood of $p$. Since the total curvature of a closed loop in $\R^n$ is at least $2\pi$, the bound $C'$ yields a uniform lower bound on $l(N_l)$. Hence, the injectivity radii of the $N_l$ are uniformly bounded from below. Therefore, there is some $r>0$ and some neighborhood $W \subset U$ of $0$ such that the injectivity radii on $W$ of $g$ and of any $g_l$ are uniformly bounded from below by $3r$ and such that the $3r$-neighborhood of $h(W)$ with respect to the induced length metric on $M$ is contained in $h(U)$. We can assume that $\frac 1 2 ||\cdot ||_{g} \leq ||\cdot ||_{g_l} \leq 2 ||\cdot ||_{g}$ for any $l$. Then the geodesic flow of any $g_l$ is well-defined as a smooth map $\phi_l : \{(t,v)\in \R \times TW \mid t ||v||_g < r \} \rightarrow TU$.

For $v \in TW$ let $\gamma_l:[0,s] \rightarrow U$ be a minimizing constant speed $g_l$-geodesic with $\gamma_l'(0)=v$ and $s ||v||_g< r $. Since $\gamma_l'$ and $\gamma_l''$ are uniformly bounded, the sequence $\gamma_l$ subconverges to a $\mathcal{C}^{1,1}$ constant speed curve $\gamma:[0,s] \rightarrow U$ with $\gamma'(0)=v$. We claim that $\gamma$ is a $g$-geodesic. For that we can assume that $\gamma$ is not constant. Then it remains to show that $d(\gamma(0),\gamma(s))\geq s ||v||_g = L(\gamma)$, where $L(\gamma)$ denotes the $g$-length of $\gamma$. Otherwise there is some $\varepsilon > 0$ and a smooth curve $c$ in $U$ connecting $\gamma(0)$ and $\gamma(s)$ of $g$-length $L(c)<L(\gamma) - \varepsilon$. Then also $L_l(c)<L(\gamma)-\varepsilon/2 <L_l(\gamma_l)-\varepsilon/3<r-\varepsilon/3$ for some sufficiently large $l$. For some larger $l$ we can extend $c$ to a piecewise smooth curve $\tilde c$ of length $L_l(\tilde c) < r$ connecting $\gamma_l(0)=\gamma(0)$ and $\gamma_l(s)$, contradicting the fact that the injectivity radius of $g_l$ on $V$ is larger than $r$. Moreover, since $M$ is a manifold with two-sided curvature bounds \cite[Proposition~1.7]{AK21} geodesics in $M$ cannot branch. Therefore, the sequence $\gamma_l$ actually converges to $\gamma$. Moreover, the geodesic flow of $g$ is well-defined as a map $\tilde \phi : \{(t,v)\in \R \times TW \mid t ||v||_g < r \} \rightarrow TU$ and the sequence of maps $\phi_l$ converges to $\tilde \phi$.

The uniform bound on the $\Pi_l$ entails a locally uniform bound on Jacobi fields via an application of Grönwall's lemma to the Jacobi equation in (\ref{eq:jacobi_geodesic_equation_sys}), which describe the derivatives of the exponential maps of the metrics $g_l$. Since the derivatives of the geodesic flows $\phi_l$ in the time direction are locally bounded as well by the geodesic equation, the limit map $\tilde \phi$ of the sequence $\phi_l$ is locally Lipschitz. Now we can choose a small neighborhood $U_p$ of $p$ so that any unit-speed $g$-geodesic starting in $U_p$ leaves $U_p$ in time $r/2$. Then $\tilde \phi$ restricts to the desired map $\phi: V \rightarrow TU_p$ and openess of $V$ follows from continuity of $\tilde \phi$.
\end{proof}

We remark that for a $\mathcal{C}^{0,1}$ Riemannian metric, like on a $\mathcal{C}^{1,1}$ submanifold of $\R^n$ as in Lemma~\ref{lem:geo_lipschitz}, constant speed locally length minimizing curves still satisfy the geodesic equation in the sense of differential inclusions of Filippov \cite{Fi88,ST18} as shown in \cite{LLS21}, see \cite[Theorem~1.1]{LLS21} and the second paragraph after \cite[Proposition~1.4]{LLS21}.

Let us now analyse the regularity of the geodesic flow of a $\mathcal{C}^{2}$ submanifold $M$ of $\R^n$. We keep working in the local setting introduced above. By Lemma \ref{lem:geo_lipschitz}, it suffices to show that the partial derivatives of the geodesic flows of smooth approximations subconverge locally uniformly, i.e. that they are locally equicontinuous. Locally in coordinates the system of the geodesic equation and the Jacobi equation of the approximating smooth submanifolds has the form
\begin{equation}\label{eq:jacobi_geodesic_equation_sys}
\begin{cases}
x_l'(t)=f_l(x_l(t)),\\
X_l'(t)=R_l(x_l(t))X_l(t),\\
x_l(0)=x_0,X_l(0)=X_0.
\end{cases}
\end{equation}
where $x$ encodes the coordinates of the geodesic and its derivative, $X$ encodes the Jacobi field and its covariant derivative and $x_0$, $X_0$ are the respective initial conditions.

Since the $\Pi_k$ converge locally uniformly to the second fundamental form $\Pi$ of $M$, they are locally equicontinuous and locally equibounded. Therefore, the coefficients in (\ref{eq:geodesic_equation_fos}) and in (\ref{eq:jacobi_geodesic_equation_sys}) locally admit a uniform bound and a uniform modulus of continuity, in particular, locally there exists a continuous nondecreasing function $\mu:[0,\infty) \rightarrow [0,\infty)$ with $\mu(0)=0$ such that $\left\| R_l(x)-R_l(y)  \right\| \leq \mu(\left\|x-y \right\|)$ for all $l$. The statement about (\ref{eq:geodesic_equation_fos}) readily implies that the time derivatives of the approximating geodesic flows, i.e. their geodesic vector fields are locally equicontinuous. To obtain the same statement for the spatial derivatives, we first recall the following lemma due to Osgood \cite[Lemma~3.4]{BCD11}\cite{Ch98}.

\begin{lem}[Osgood's lemma] \label{lem:osgood} Let $L$ be a measurable nonnegative function defined on $[t_0,t_1]$. Let $\mu:[0,\infty) \rightarrow [0,\infty)$ be a continuous nondecreasing function with $\mu(0)=0$. Let $a\geq 0$, and assume that for a.e. $t$ in $[t_0,t_1]$
\[
			L(t) \leq a + \int_{t_0}^{t} \mu(L(s)) ds.
\]
If $a>0$, then 
\[
			\int_a^{L(t)} \frac{ds}{\mu(s)} \leq t-t_0
\]
for  a.e. $t \in [t_0,t_1]$. If $a=0$ and $\int_0^{\infty} ds/\mu(s)= \infty$, then $L\equiv 0$  a.e.. Moreover, if $L$ is continuous, then the conclusions holds everywhere.
\end{lem}

For an ODE of the form $y'=g(t,y(t))$ where $g$ admits a so-called Osgood modulus of continuity in $y$, i.e. a modulus of continuity $\mu:[0,\infty) \rightarrow [0,\infty)$ satisfying the Osgood condition
\[
				\int_0^1 \frac{dx}{\mu(x)} = \infty,
\]
a standard application of Osgood's lemma yields a unique solution $y(t,y_0)$ of the initial value problem $y'=g(t,y(t))$, $y(0)=y_0$ which is continuous in $y_0$ and whose modulus of continuity can be expressed in terms of $\mu$ \cite[Chapter~3.1]{BCD11}. While the Osgood condition is not satisfied for the system (\ref{eq:jacobi_geodesic_equation_sys}) it is satisfied for the Jacobi equation. The special structure of (\ref{eq:jacobi_geodesic_equation_sys}) allows us to first solve the geodesic equation and then the Jacobi equation with the base coordinates as parameters. It remains to extend the application of Osgood's lemma to this setting.

\begin{lem}\label{lem:c2_case} 
Consider a linear ODE of the form $X'(t)=R(t,x_0) X(t)$ for $t\in [0,t_1]$, where the matrix of coefficients $R$ depends on time and on some parameter $x_0 \in U \subset \R^n$. If $R$ is bounded by $\bar C$ and continuous with a modulus of continuity $\mu_R$, then solutions $X(t,x_0,X_0)$ of the initial value problem
\begin{equation}\label{eq:jacobi_geodesic_equation_sys2}
\begin{cases}
X'(t)=R(t,x_0) X(t),\\
X(0)=X_0,
\end{cases}
\end{equation}
with $||X_0||<C$ for some fixed $C>0$ are also continuous with a modulus of continuity that only depends on $C$, $\bar C$, $\mu$ and $t_0$. Moreover, if $R$ is $\alpha$-Hölder continuous for some $\alpha \in (0,1]$ with $\mathcal{C}^{0,\alpha}$ norm $\bar C$, then solutions $X(t,x_0,X_0)$ of the initial value problem (\ref{eq:jacobi_geodesic_equation_sys2}) are $\alpha$-Hölder continuous as well with a $\mathcal{C}^{0,\alpha}$ norm that only depends on $C$, $\bar C$ and $t_0$.
\end{lem}

\begin{proof} For $t,t' \in [0,t_1]$, $x_0,x_0' \in U$ and initial conditions $X_0,X_0'$ with $||X_0||, ||X_0'|| < C$ we estimate the difference $X(t,x_0,X_0)-X(t',x_0',X_0')$:
\begin{equation}\label{eq:two_terms}
\begin{split}
  \left\| X(t,x_0,X_0)-X(t',x_0',X_0')  \right\| \leq &\left\| X(t,x_0,X_0)-X(t',x_0,X_0)  \right\| \\
													&+\left\| X(t',x_0,X_0)-X(t',x_0',X_0)  \right\| \\
													&+ \left\| X(t',x_0',X_0)-X(t',x_0',X_0')  \right\|.
\end{split}
\end{equation}
Recall that $(\ref{eq:jacobi_geodesic_equation_sys2})$ locally (in $X_0$) yields a uniform bound $\tilde C$ on the solutions $X$ as in the proof of Lemma~\ref{lem:geo_lipschitz} via Grönwall's lemma. Then (\ref{eq:jacobi_geodesic_equation_sys2}) implies a local bound on the first summand in (\ref{eq:two_terms}) of the form $C_1|t-t'|$ for some constant $C_1>0$. Moreover, since the Jacobi equation is linear in $X$ a standard application of Osgood's lemma yields a uniform (in $t'$ and $x_0'$) modulus of continuity for the dependence on $X_0$ in the third summand \cite[Chapter~3.1]{BCD11}. We illustrate this application of Osgood's lemma by finding a bound for the second summand. For any $t \in [0,t_1]$ we have
\begin{equation*}
\begin{split}
&  \left\| X(t,x_0,X_0)-X(t,x_0',X_0)  \right\| \\
 & \leq\left\| \int_0^{t} \left(R(s,x_0)X(s,x_0,X_0) -R(s,x_0')X(s,x_0',X_0) \right) ds   \right\| \\
 &\leq  \int_0^{t}  \left\| R(s,x_0) \right\| \left\| X(s,x_0,X_0) - X(s,x'_0,X_0) \right\|  ds  \\
 & +  \int_0^{t} \left\| R(s,x_0) - R(s,x'_0) \right\|  \left\| X(s,x'_0,X_0)\right\| ds   \\
 &\leq  a + \int_0^{t}  \bar{C} \left\| X(s,x_0,X_0) - X(s,x'_0,X_0) \right\|  ds  \\ 
\end{split}
\end{equation*}
with $a=   \tilde C |t_1| \left\| R(\cdot,x_0) - R(\cdot,x'_0) \right\|_{[0,t_1]}$. 

For $a=0$ Osgood's lemma implies 
\[
X(t,x_0,X_0)=X(t,x_0',X_0)
\]
for all $t \in [0,t_1]$.

For $a>0$ Osgood's lemma yields $\int_a^{L} \frac{dr}{Cr} \leq |t_1|$ with $L=\left\| X(t_1,x_0,X_0) - X(t_1,x'_0,X_0) \right\|$. For $||x_0 -x_0'||\leq \delta$ we have $a \leq \tilde{C} |t_1| \mu_R(\delta)$ and hence $\int_{\tilde C |t_1|\mu_R(\delta)}^{L} \frac{dr}{Cr} \leq |t_1|$. We define $\Gamma:[0,\infty) \rightarrow [0,\infty)$ by $\Gamma(0)=0$ and $\int_{\tilde C |t_1|\mu_R(\delta)}^{\Gamma(\delta)} \frac{dr}{Cr} = |t_1|$. Then $\Gamma$ is a uniform (in $t \leq t_1$ and $X_0$ for $||X_0||<C$) modulus of continuity for $X$ depending on $x_0$ and it is explicitly given by 
\[
		\Gamma(\delta) = \tilde C |t_1| e^{C|t_1|} \mu_R(\delta).
\]
In total we find a desired modulus of continuity of $X$ for the dependence on $t$, $x_0$ and $X_0$, which completes the proof of the first claim.

If $R$ is $\alpha$-Hölder continuous for some $\alpha \in (0,1]$ with $\mathcal{C}^{0,\alpha}$-norm $\bar C$, then the modulus of continuity $\mu_R$ in the above proof can be chosen to be of the form $\mu_R(\delta)= \bar C \delta^{\alpha}$. In this case $\Gamma$ is proportional to $\delta^{\alpha}$ as well. The same argument yields a uniform $\alpha$-Hölder modulus of continuity for the dependence of $X(x_0,X_0)$ on $X_0$, cf. the third term of (\ref{eq:two_terms}). The dependence on $t$ is Lipschitz and hence also $\alpha$-Hölder. Again, in total we find a desired $\alpha$-Hölder modulus of continuity of $X$ for the dependence on $x_0$ and $X_0$.
\end{proof}

Now, since the $R_l(t,x_0) = R_l(x_l(t,x_0))$ are locally equicontinuous and equibounded an application of Lemma \ref{lem:c2_case}  completes the proof of Theorem \ref{thm:regularity_gf} according to the discussion preceding Lemma \ref{lem:osgood}.

\begin{rem}\label{rem:hoelder_case} Based on Lemma \ref{lem:c2_case} one can also show the Lipschitz resp. Hölder version of Theorem \ref{thm:regularity_gf} stated in the introduction. However, in the Hölder case $\alpha<1$ the claim does not localize as before as the composition of two $\alpha$-Hölder continuous functions is in general not $\alpha$-Hölder continuous anymore. Instead, one can work in a global immersed tubular neighborhood of a given geodesic.
\end{rem}


\begin{thebibliography}{[O'Ne66]}

\bibitem[BCD11]{BCD11} H. Bahouri, J. Chemin and R. Danchin, \emph{Fourier analysis and nonlinear partial differential equations}. Grundlehren der mathematischen Wissenschaften, 343. Springer, Heidelberg, 2011. xvi+523 pp.

\bibitem[Ch98]{Ch98} J.~Chemin, Perfect incompressible fluids. Oxford Lecture Series in Mathematics and its Applications, 14. The Clarendon Press, \emph{Oxford University Press}, New York, 1998. x+187 pp

\bibitem[DK81]{DK81} D.~M. DeTurck and J.~L. Kazdan, Some regularity theorems in Riemannian geometry, \emph{Ann. Sci. École Norm. Sup.} (4) 14 (1981), no. 3, 249–260. 

\bibitem[dC92]{dC92} M.~P. do Carmo, \emph{Riemannian geometry}. Mathematics: Theory \& Applications. Birkhäuser Boston, Inc., Boston, MA, 1992.

\bibitem[Fi88]{Fi88} A. F. Filippov. Differential equations with discontinuous righthand sides, volume 18 of Mathematics and its Applications (Soviet Series). \emph{Kluwer Academic Publishers Group}, Dordrecht, 1988. 

\bibitem[Hal77]{Hal77} B. Halpern, Strange billiard tables, \emph{Trans. Amer. Math. Soc.} {\bf 232} (1977), 297--305.

\bibitem[Ha82]{Ha82} P.~Hartman, \emph{Ordinary differential equations}. Classics in Applied Mathematics, 38. Society for Industrial and Applied Mathematics (SIAM), Philadelphia, PA, 2002. xx+612 pp.

\bibitem[Ha50]{Ha50} P.~Hartman, On the local uniqueness of geodesics. \emph{Amer. J. Math.} 72 (1950), 723–730.

\bibitem[Ha51]{Ha51} P.~Hartman, On geodesic coordinates. \emph{Amer. J. Math.} 73 (1951), 949–954. 

\bibitem[La22]{La22} C.~Lange, On continuous billiard and quasigeodesic flows characterizing alcoves and isosceles tetrahedra, \emph{J. Lond. Math. Soc.},  https://doi.org/10.1112/jlms.12724, (2022).

\bibitem[LLS21]{LLS21} C.~Lange, A.~Lytchak and C.~Sämann, Lorentz meets Lipschitz. \emph{Adv. Theor. Math. Phys.} 25 (2021), no. 8, 2141–2170.

\bibitem[AK21]{AK21} V.~Kapovitch and A.~Lytchak, Remarks on manifolds with two-sided curvature bounds. \emph{Anal. Geom. Metr. Spaces} 9 (2021), no. 1, 53–64. 

\bibitem[KSS14]{KSS14} M.~Kunzinger, R.~Steinbauer and M.~Stojković, The exponential map of a $C^{1,1}$-metric. \emph{Differential Geom. Appl.} 34 (2014), 14–24. 

\bibitem[AY06]{AY06} A.~Lytchak and A.~Yaman, On Hölder continuous Riemannian and Finsler metrics. \emph{Trans. Amer. Math. Soc.}, 358(7): 2917–2926, 2006.

\bibitem[MMS14]{MMS14} M. G. Monera, A. Montesinos-Amilibia\ and\ E. Sanabria-Codesal, The Taylor expansion of the exponential map and geometric applications, \emph{Rev. R. Acad. Cienc. Exactas F\'{\i}s. Nat. Ser. A Mat. RACSAM} {\bf 108} (2014), no.~2, 881--906.


\bibitem[Min14]{Min14} E.~Minguzzi, The equality of mixed partial derivatives under weak differentiability conditions. \emph{Real Anal. Exchange} 40 (2014/15), no. 1, 81–97. 

\bibitem[Min15]{Min15} E.~Minguzzi, Convex neighborhoods for Lipschitz connections and sprays. \emph{Monatsh. Math.} 177 (2015), no. 4, 569–625.

\bibitem[Na54]{Na54} J.~Nash, $C^1$ isometric imbeddings. \emph{Ann. of Math.} (2) 60 (1954), 383–396. 

\bibitem[Na56]{Na56} J.~Nash, The imbedding problem for Riemannian manifolds. \emph{Ann. of Math.} (2) 63 (1956), 20–63. 

\bibitem[Pe92]{Pe92} G. Peano. Sur la définition de la dérivée. \emph{Mathesis}, 2 (1892) 12–14.

\bibitem[Pe16]{Pe16} P.~Petersen, \emph{Riemannian geometry}. Third edition. Graduate Texts in Mathematics, 171. Springer, Cham, 2016. xviii+499 pp.

\bibitem[ST18]{ST18} C.~Sämann and R. Steinbauer,  On geodesics in low regularity. \emph{J. Phys. Conf. Ser.} 968 (2018), 012010, 14 pp.

\bibitem[Ru76]{Ru76}  W.~Rudin, \emph{Principles of mathematical analysis}. Third edition. International Series in Pure and Applied Mathematics. McGraw-Hill Book Co., New York-Auckland-Düsseldorf, 1976.

\bibitem[SS76]{SS76} I. H.~Sabitov and S.Z.~Šefel, Connections between the order of smoothness of a surface and that of its metric. (Russian) \emph{Sibirsk. Mat. Ž.} 17 (1976), no. 4, 916–925.

\bibitem[Sa93]{Sa93} Sabitov, I. Kh., On the smoothness of isometries. \emph{Siberian Math.} J. 34 (1993), no. 4, 741–748.

\bibitem[Wa83]{Wa83} F. W. Warner, \emph{Foundations of differentiable manifolds and Lie groups}. Graduate Texts in Mathematics, 94. Springer-Verlag, New York-Berlin, 1983.

\end{thebibliography}
\end{document}